\documentclass[11pt,leqno,]{article}
\usepackage[utf8]{inputenc}
\usepackage{a4}
\usepackage[T1]{fontenc}
\usepackage[english]{babel}
\usepackage{amssymb,amsmath,amsthm}
\usepackage{bbold}
\usepackage{enumerate}
\usepackage{graphicx}
\usepackage{epsfig}
\usepackage{comment}
\usepackage{xspace}
\usepackage{multirow}
\usepackage{color}
\usepackage{nomencl}

\makenomenclature
\newcommand{\loc}{\ensuremath{\text{loc}}\xspace}

\newcommand{\n}{\ensuremath{\nabla}\xspace}
\newcommand{\ud}{\ensuremath{u_{\delta}}\xspace}
\newcommand{\N}{\ensuremath{\mathbb{N}}\xspace}
\newcommand{\R}{\ensuremath{\mathbb{R}}\xspace}

\renewcommand{\epsilon}{\varepsilon}

\begin{document}

\numberwithin{equation}{section}
\newtheoremstyle{break}{15pt}{15pt}{\itshape}{}{\bfseries}{}{\newline}{}
\theoremstyle{break}
\newtheorem*{Satz*}{Theorem}
\newtheorem*{Rem*}{Remark}
\newtheorem*{Lem*}{Lemma}
\newtheorem{Satz}{Theorem}[section]
\newtheorem{Rem}{Remark}[section]
\newtheorem{Lem}{Lemma}[section]
\newtheorem{Cor}{Corollary}[section]
\newtheorem{Prop}{Proposition}[section]
\theoremstyle{definition}

\parindent2ex

\begin{center}
{\Large \bf Existence and almost everywhere regularity of generalized minimizers for a class of variational problems with linear growth related to image inpainting}
\end{center}

\begin{center}
 J. M\"uller\footnote{The first author would like to express his thankfulness to  Joachim Weickert for providing financial support.} \& C. Tietz\footnote{Both authors want to thank Michael Bildhauer, Martin Fuchs and Joachim Weickert for  many inspiring and stimulating discussions.}
\end{center}


\begin{abstract}
We continue the analysis of some modifications of the total variation image inpainting method formulated on the space $BV(\Omega)^M$ in the sense that we generalize the main results of \cite{FT} to the case that a more general data fitting term is involved. As in \cite{FT} we deal with vector-valued images, we do not impose any structure condition on our density $F$ and the dimension of the domain $\Omega$ is arbitrary. Precisely we discuss existence of generalized solutions of the corresponding variational problem and we will also pass to the associated dual variational problem for which we show unique solvability. Among other things, our results are the uniqueness of the absolutely continuous part $\nabla^a u$ of the gradient of $BV$-solutions $u$ on the entire domain $\Omega$, where outside of the damaged region $D$ we even get uniqueness of $BV$-solutions. Imposing stronger assumptions on our density $F$ and an $L^{\infty}$-condition on our partial observation $f$ we are going to prove a maximum principle for each generalized minimizer and deduce partial $C^{1,\beta}$-regularity of solutions on the entire domain $\Omega$ for all $0<\beta\leq\frac{1}{2}$.
\end{abstract}

\begin{section}{Introduction}
In this article we investigate the existence and the regularity of generalized solutions for a variational problem being related to variational and PDE-based methods that are frequently used in image recovery. Particularly we are concerned with a modified variant of the TV-regularization. Without being entitled of being complete we refer to \cite{AV, AVe, BKP, BF0, CCN, CL, CE, CSV, Ka, ROF, V} and the references quoted therein where the TV-regularization and some related models with superlinear growth have been studied extensively. Among theoretical aspects, the study partially also involved some numerical issues.\\
In our note we concentrate on a modified variant of the so-called total variation image inpainting method where our paramount intention is to generalize the results from \cite{FT} to the case of more general data fitting terms under consideration. Here we take the case of penalty terms including the general $L^p$-norm for $1\leq p<\infty$ as a model. Albeit our prior impetus is of rather theoretical nature it seems that at least in the case of pure denoising of images, one is interested in studying reasonable modifications of the so-called $TV-L^p$-regularization. If we set $p=2$ in this context we deal with the prominent Rudin-Osher-Fatemi model (see, e.g., \cite{SWT} or \cite{ROF} for the original paper) and this model seems to be predestinated in order to remove white additive Gaussian noise (see, e.g., \cite{AK}, pp. 62). Choosing $p\neq2$, the corresponding model seems to be useful for object recognition and image segmentation (see, e.g., \cite{SWT}) while in the limit case $p=1$, the $TV-L^1$-regularization seems to be rather reasonable for removing impulsive noise (see, e.g., \cite{GY}).\\
Our work concentrates on the concept of image inpainting that we briefly illustrate in the following:\\
suppose that we are given a bounded Lipschitz domain $\Omega$ in $\R^n$ with $n\geq2$ (e.g. a rectangle in the case $n=2$ or a cuboid in the case $n=3$) and an $\mathcal{L}^n$-measurable subset $D$ of $\Omega$ ($\mathcal{L}^n$ denoting Lebesgue's measure on $\R^n$) satisfying
\begin{align}
 \label{L} 0\leq\mathcal{L}^n(D)<\mathcal{L}^n(\Omega).
\end{align}
We further assume that we are given an observed image described through a measurable function $f:\Omega-D\rightarrow\R^M$ for which we require
\begin{align}
\label{f}  f\in L^{\zeta}(\Omega-D)^M,
\end{align}
where $\zeta>1$ denotes a fixed, finite number. Roughly speaking, the \grqq inpainting domain\grqq\,\, $D$ (compare \cite{BHS}) represents a certain part of this image for which image data are missing or inaccessible and our aim is to restore this missing part from the part which is known, i.e. to generate an image $u:\Omega\rightarrow\R^M$ based on the partial observation $f:\Omega-D\rightarrow\R^M$.\\
\newline
As outlined in \cite{FT} we can distinguish between various types of images depending on the dimension of the domain or of the codomain, respectively. In case $n=2, M=1$ we are concerned with a classical digital image (see, e.g., \cite{AK, LOPR}) while in the case $n=3, M=1$ we deal with three-dimensional images that are of fundamental meaning in medical imaging, e.g. computerized tomography or magnetic resonance imaging (see, e.g., \cite{LGS, LOPR} and the references quoted therein). Considering the case $M\geq2$ we are confronted with vector-valued images (e.g. coloured images where each channel (or dimension) illustrates a corresponding colour (see, e.g., \cite{BC} and the references quoted therein)).\\
The kind of image interpolation described above, at least in the case $n=2$ and $M=1$, is called \grqq inpainting\grqq\,\, or \grqq image inpainting\grqq\,\,, respectively (compare \cite{BHS,PSS, Sh}). In literature there are a lot of different techniques in order to handle the inpainting problem being of variational or non-variational as well as of local or non-local kind where without being complete we mention \cite{ACFLS,BBCS, BHS,CKS,CS1,CS2, ES, PSS, Sh} and the references quoted therein. In what follows we discuss a TV-like variational approach being of non-local type as proposed in \cite{BF1, BF2, BF3, BF5, BFT} and also in the related work \cite{BF0}. Precisely we seek minimizers of the functional
\begin{align}
 \label{J} J[w]:=\int\limits_{\Omega}{\psi(|\nabla w|)dx}+\frac{\lambda}{\zeta}\int\limits_{\Omega-D}{|w-f|^{\zeta} dx}
\end{align}
where $\lambda$ is a positive regularization parameter and $\psi$ is supposed to be a convex and increasing function with non-negative values.\\
The second term on the right-hand side of \eqref{J} measures the quality of data fitting, i.e. the deviation of the original image $u$ from the given data $f$ on $\Omega-D$ while the first term allows to incorporate some kind of apriori information of the generated image via some kind of mollification on the entire domain $\Omega$ into the minimization process.\\
In this setup, a common choice of $\psi$ is $\psi(|\nabla w|):=|\nabla w|$. This leads to the total variation inpainting model (compare \cite{ACS,PSS}). In order to study this variational problem, one has to work with functions $\Omega\rightarrow\R^M$ of bounded variation, i.e. in the space $BV(\Omega)^M$. This space covers all $L^1$-functions whose distributional gradient $\nabla w$ is represented by a tensor-valued Radon measure on $\Omega$ with finite total variation $\int\limits_{\Omega}{|\nabla w|}$ (for details, we refer to \cite{AFP} or \cite{Giu}).\\
In this paper we seize the basic idea carried out in \cite{BF0}, i.e. we replace the unpleasant TV-density $|\n w|$ by an integrand $F(\n w)$ of linear growth being strictly convex w.r.t. the tensor-valued Radon measure $\n w$ and study solvability of the corresponding variational problem
\begin{align}
\label{Ibv} 
\int\limits_{\Omega}{F(\nabla w)}+\frac{\lambda}{\zeta}\int\limits_{\Omega-D}{|w-f|^{\zeta} dx}\rightarrow\text{min}\quad\text{in}\,\,BV(\Omega)^M\cap L^{\zeta}(\Omega-D)^M.
\end{align}
In addition we will pass to the dual variational problem for which we prove unique solvability under weak assumptions. We further like to mention that in case $\zeta=2$, problem \eqref{Ibv} has been studied extensively (see \cite{BF1, BF2, BF3, BF5, BFT} for the case $n=2, M=1$, \cite{FT, T} for any dimension $n$ and arbitrary codimension $M$ as well as \cite{BF0} for the analysis of pure denoising in the special case $n=2$ and with arbitrary codimension $M$ where additional boundary data could be included).\\
\newline
At this stage we like to make a few comments concerning the space in which we discuss problem \eqref{Ibv}: on account of the continuous embedding $BV(\Omega)^M\hookrightarrow L^{\frac{n}{n-1}}(\Omega)^M$ being valid for \grqq bounded extension domains\grqq\,\,$\Omega$ (see, e.g., \cite{AFP}, Corollary 3.4, p.152) and considering numbers $\zeta>\frac{n}{n-1}$ the requirement \grqq $w\in L^{\zeta}(\Omega-D)^M$\grqq\,\,is an additional constraint. If we choose $\zeta\leq\frac{n}{n-1}$ we may discuss problem \eqref{Ibv} just in the entire space $BV(\Omega)^M$.\\
\newline
Now, let us fix our setup and state our precise hypotheses: suppose that we are given a function $F:\R^{nM}\rightarrow[0,\infty)$ being of class $C^1(\R^{nM})$ satisfying the following assumptions
\begin{align}
\label{v1} &\text{F is strictly convex}, F(0)=0,\\
\label{v2} &|DF(P)|\leq\nu_1,\\
\label{v3} &F(P)\geq\nu_2|P|-\nu_3
\end{align}
with constants $\nu_1,\nu_2>0,\nu_3\in\R$, for all $P\in\R^{nM}$. In accordance with \eqref{v2} and $F(0)=0$ we directly get
\[
 F(P)\leq\nu_1|P|
\]
for all $P\in\R^{nM}$, i.e. $F$ is of linear growth in the following sense
\begin{align}
 \label{lg} \nu_2|P|-\nu_3\leq F(P)\leq\nu_1|P|.
\end{align}
As a starting point we consider the following problem
\begin{align}
 \label{vp1} \begin{split} I[w]:=&\int\limits_{\Omega}{F(\nabla w)dx}+\frac{\lambda}{\zeta}\int\limits_{\Omega-D}{|w-f|^{\zeta} dx}\rightarrow\text{min}\\
&\text{in}\,\,W^{1,1}(\Omega)^M\cap L^{\zeta}(\Omega-D)^M.
\end{split}
\end{align} 
Since $F$ is of linear growth, the Sobolev space $W^{1,1}(\Omega)^M$ (see, e.g., \cite{Ad} for details concerning these spaces) acts as the suitable Sobolev space in order to discuss problem \eqref{vp1}. Unfortunately, for lack of reflexivity, we can not expect solvability of \eqref{vp1} in this space where we like to mention that in the case $\zeta=2$ and under rather strong assumptions we can discuss problem \eqref{vp1} in $W^{1,1}(\Omega)^M$ without passing to a relaxed variant in the space $BV(\Omega)^M$ (see \cite{BF1} for the case $n=2, M=1$ and \cite{T} for any dimension $n$ and arbitrary codimension $M$). Thus, the question arises how to deal with problem \eqref{vp1} in general. A natural and established approach is to apply the concept of convex functions of a measure (see, e.g., \cite{AG, DT} or \cite{GMS1}), i.e. we let for $w\in BV(\Omega)^M\cap L^{\zeta}(\Omega-D)^M$
\begin{align}
 \label{K} K[w]:=\int\limits_{\Omega}{F(\nabla^a w)dx}+\int\limits_{\Omega}{F^{\infty}\bigg(\frac{\nabla ^s w}{|\nabla ^s w|}\bigg)d|\nabla^s w|}+\frac{\lambda}{\zeta}\int\limits_{\Omega-D}{|w-f|^{\zeta}dx}.
\end{align}
Here, we denote for tensor-valued Radon measures $\rho$ by $\rho^a(\rho^s)$ the regular (singular) part of $\rho$ w.r.t. to Lebesgue's measure $\mathcal{L}^n$. Moreover, for a convex function $F:\R^{nM}\rightarrow[0,\infty)$, the recession function $F^{\infty}:\R^{nM}\rightarrow[0,\infty]$ of $F$ is defined by
\begin{align}
\label{recession} F^{\infty}(P):=\lim\limits_{t\rightarrow\infty}{\frac{F(tP)}{t}},\quad P\in\R^{nM}.
\end{align}
Since $F$ is (strictly) convex and of linear growth (recall \eqref{lg}) it follows that $F^{\infty}$ is well-defined.
Now, the idea is to seek minimizers of the relaxed variational problem
\begin{align}
 \label{Kv} K\rightarrow\text{min}\quad\text{in}\,\, BV(\Omega)^M\cap L^{\zeta}(\Omega-D)^M
\end{align}
and to introduce them as generalized solutions of \eqref{vp1}.\\
\newline
At this point, we will state our first theorem that proves solvability of problem \eqref{Kv}. Besides we will show uniqueness of the absolutely continuous part $\n^a u$ of the gradient of $BV$-solutions on the entire domain $\Omega$ and will additionally verify the uniqueness of $BV$-solutions outside of the damaged region $D$. In part (c) we justify that each $K$-minimizer can be introduced as a generalized minimizer of the original functional $I$ while in part (d) we show that each $K$-minimizer belongs to the set $\mathcal{M}$ of generalized minimizers of the functional $I$ from \eqref{vp1} and vice versa. In case $\zeta=2$ these results can be found in \cite{FT}, Theorem 1.1.
\begin{Satz}\label{BV}
Let us fix a real number $\zeta>1$ and further we assume the validity of \eqref{L} as well as \eqref{f}. Moreover we let $F$ satisfy \eqref{v1}--\eqref{v3}. Then it holds:
\begin{enumerate}[(a)]
 \item Problem \eqref{Kv} has at least one solution.
\item Suppose that $u$ and $\widetilde{u}$ are $K$-minimizing. Then
\[
 u=\widetilde{u}\,\,\text{a.e. on}\,\,\Omega-D\quad\text{and}\quad\nabla^au=\nabla^a\widetilde{u}\,\,\text{a.e. on}\,\,\Omega.
\]
\item It holds
\[
 \inf\limits_{W^{1,1}(\Omega)^M\cap L^{\zeta}(\Omega-D)^M}{I}=\inf\limits_{BV(\Omega)^M\cap L^{\zeta}(\Omega-D)^M}{K}.
\]
\item Let $\mathcal{M}$ denote the set of all $L^1(\Omega)^M$-cluster points of $I$-minimizing sequences from the space $W^{1,1}(\Omega)^M\cap L^{\zeta}(\Omega-D)^M$. Then $\mathcal{M}$ coincides with the set of all $K$-minimizers from the space $BV(\Omega)^M\cap L^{\zeta}(\Omega-D)^M$.
\end{enumerate}
\end{Satz} 

\begin{Rem}
Following the lines of the proof of Theorem 1.1 in \cite{FT} for the case $\zeta=2$ it becomes evident that Theorem \ref{BV} extends to any finite $\zeta>1$. Here it should be emphasized that in case $\zeta>\frac{n}{n-1}$ we have to make use of the density result for functions of $BV$-type that has been formulated in \cite{FT} (see Lemma 2.2 in this reference). For the reader's convenience we reformulate and prove this density result in Section 2 of the present paper.
\end{Rem}

\begin{Rem}
\begin{itemize}
 \item Theorem \ref{BV} also extends to more general strictly convex and strictly increasing data fitting terms $h:[0,\infty)\rightarrow[0,\infty)$ with $h(0)=0$. For instance we can choose 
\begin{align*}
&h(|w-f|):=|w-f|\log(1+|w-f|)\quad\text{or}\\
&h(|w-f|):=\Phi_{\nu}(|w-f|),\quad\nu>1,
\end{align*}
where $\Phi_{\nu}$ denotes the integrand from \eqref{Phimu} below.
\item Note that Theorem \ref{BV} partially extends to the case $\zeta=1$. In fact we a priori do not get uniqueness of solutions on the set $\Omega-D$.
\end{itemize}
\end{Rem}

\begin{Rem}\label{remark2}
\begin{itemize}
 \item Part (b) of Theorem \ref{BV} shows uniqueness of solutions on the set $\Omega-D$ and the measures $\nabla u$ and $\nabla\widetilde{u}$ of minima $u,\widetilde{u}$ may only differ in their singular parts.
\item The statements (c) and (d) in Theorem \ref{BV} reveal that the minimization of $K$ in $BV(\Omega)^M\cap L^{\zeta}(\Omega-D)^M$ represents a natural extension of the original variational problem \eqref{vp1} which in general fails to have a solution in $W^{1,1}(\Omega)^M\cap L^{\zeta}(\Omega-D)^M$. Furthermore, it holds $I=K$ on $W^{1,1}(\Omega)^M\cap L^{\zeta}(\Omega-D)^M$ and this fact also stresses that we have a reasonable extension of the functional $I$. 
\end{itemize}
\end{Rem}

\begin{Rem}\label{Voraussetzungen}
The assumptions on our density $F$ in Theorem \ref{BV} can be weakened in such a way that we require that $F$ is strictly convex and of linear growth in the sense of \eqref{lg}.
\end{Rem}

From assertion (b) of Theorem \ref{BV} we can deduce the uniqueness in case of $W^{1,1}$-solvability. Furthermore, in the general case, the $L^{\frac{n}{n-1}}$-deviation $\|u-v\|_{L^{\frac{n}{n-1}}(\Omega)}$ of different solutions $u,v$ on the damaged region $D$ can be estimated in terms of $\n^s(u-v)$. In case $\zeta=2$, the statements below have been derived in \cite{FT}, Corollary 1.1. 
\begin{Cor}\label{corollary}
With the notation and the assumptions of Theorem \ref{BV} it holds
\begin{enumerate}[(i)]
 \item If there exists an element $u\in\mathcal{M}$ such that $u\in W^{1,1}(\Omega)^M\cap L^{\zeta}(\Omega-D)^M$, then it holds $\mathcal{M}=\{u\}$.
\item Suppose that $\overline{D}\subset\Omega$. Then there exists a constant $c=c(n,M)$ such that we have for all $u,v\in\mathcal{M}$
\[
 \|u-v\|_{L^{\frac{n}{n-1}}(\Omega)}= \|u-v\|_{L^{\frac{n}{n-1}}(D)}\leq  c|\n^s(u-v)|(\overline{D})
\]
where, in particular, the constant $c$ on the right-hand side does not depend on the free parameter $\lambda$.
\end{enumerate}
\end{Cor}

\begin{Rem}
As outlined in \cite{FT}, Remark 1.1, a proof of Corollary \ref{corollary} can be deduced by observing that Corollary 1.1 from \cite{BF2} extends to any dimension $n\geq2$ and remains also valid for vector-valued functions, i.e. for the case $M\geq2$. The corresponding references can be found in \cite{BF2}, proof of Corollary 1.1. Furthermore, Corollary \ref{corollary}, (ii) can be interpreted in such a way that the $L^{\frac{n}{n-1}}$-deviation $\|u-v\|_{L^{\frac{n}{n-1}}}$ of different solutions $u,v$ on the inpainting region $D$ is governed by the total variation of the singular part $\nabla^s(u-v)$ of the tensor-valued Radon measure $\nabla(u-v)$.
\end{Rem}

Considering selected variational problems in physics (e.g. in the theory of perfect plasticity, see \cite{FS} for a survey), another approach to problem \eqref{vp1} seems to be more natural. An essential motivation for studying dual variational problems is the uniqueness of the corresponding maximizer $\sigma$ under rather weak assumptions (for more detailed information we refer to section 2.2 in \cite{Bi1}). Moreover, the dual solution $\sigma$ usually admits a clear geometric or physical interpretation. For instance, we can note that in the theory of minimal surfaces, $\sigma$ corresponds to the normal of the surface and in the theory of plasticity, $\sigma$ represents the stress tensor. However we do not know an adequate interpretation of the dual solution $\sigma$ in the context of image processing.\\
\newline
Let $F$ satisfy \eqref{v1}--\eqref{v3} and suppose that \eqref{L} as well as \eqref{f} hold. Further we fix a real number $\zeta>1$. Following \cite{ET} we define the Lagrangian for functions $w\in W^{1,1}(\Omega)^M\cap L^{\zeta}(\Omega-D)^M$ and $\varkappa\in L^{\infty}(\Omega)^{nM}$ by the equation
\[
 l(w,\varkappa):=\int\limits_{\Omega}{[\varkappa :\n w-F^*(\varkappa)]dx}+\frac{\lambda}{\zeta}\int\limits_{\Omega-D}{|w-f|^{\zeta}dx},
\]
where
\[
 F^*(Q):=\sup\limits_{P\in \R^{nM}}{[P:Q-F(P)]},\quad Q\in\R^{nM},
\]
represents the conjugate function to $F$. By means of \cite{ET}, Proposition 2.1, p.271, we obtain the representation
\begin{align}
 \label{representationdual} \int\limits_{\Omega}{F(P)dx}=\sup\limits_{\varkappa\in L^{\infty}(\Omega)^{nM}}{\int\limits_{\Omega}{[\varkappa:P-F^{*}(\varkappa)]dx}}
\end{align}
for functions $P\in L^1(\Omega)^{nM}$ leading to the following alternative formula for the functional $I$
\begin{align}
 \label{Ineu} I[w]=\sup\limits_{\varkappa\in L^{\infty}(\Omega)^{nM}}{l(w,\varkappa)},\quad w\in W^{1,1}(\Omega)^M\cap L^{\zeta}(\Omega-D)^M,
\end{align}
and by virtue of \eqref{Ineu} we can introduce the dual functional
\begin{align*}
 R:L^{\infty}(\Omega)^{nM}&\rightarrow[-\infty,\infty],\\
R[\varkappa]&:=\inf\limits_{w\in W^{1,1}(\Omega)^M\cap L^{\zeta}(\Omega-D)^M}{l(w,\varkappa)}.
\end{align*}
Consequently, the dual problem is: to maximize $R$ among all functions $\varkappa\in L^{\infty}(\Omega)^{nM}$.\\
\newline
In what follows we would like to present our results on the dual variational problem associated to \eqref{vp1}. In case $\zeta=2$, the results stated below can be found in \cite{FT}, Theorem 1.2.
\begin{Satz}\label{dual}
Suppose that we are given a real number $\zeta>1$. Further let \eqref{L} and \eqref{f} hold and assume that $F$ satisfies \eqref{v1}--\eqref{v3}. Then we have:
\begin{enumerate}[(a)]   
\item The dual problem
\[
 R\rightarrow\text{max}\quad\text{in}\,\,L^{\infty}(\Omega)^{nM}
\]
admits at least one solution. Moreover, the $\inf$-$\sup$ relation 
\[
 \inf\limits_{v\in W^{1,1}(\Omega)^M\cap L^{\zeta}(\Omega-D)^M}{I[v]}=\sup\limits_{\sigma\in L^{\infty}(\Omega)^{nM}}{R[\sigma]}
\]
is valid.
\item We have uniqueness of the dual solution if the conjugate function $F^{*}$ is strictly convex on the set $\{p\in\R^{nM}:F^{*}(p)<\infty\}$. 
\end{enumerate}
\end{Satz}

\begin{Rem}
Considering the structure of the data fitting term $\int\limits_{\Omega-D}{|w-f|^{\zeta}\,dx}$ in case $\zeta>1,\zeta\neq2$, we emphasize that we are confronted with severe difficulties. As a consequence we can not immediately refer to the proof of Theorem 1.2 in \cite{FT} adding some obvious modifications in the case $\zeta>1,\zeta\neq2$.
\end{Rem}

In order to derive uniqueness of the dual solution $\sigma$, the imposed hypothesis on $F^*$ in Theorem \ref{dual} (b) can be dropped. Further, the unique dual solution is related to the $BV$-solutions from Theorem \ref{BV} through an equation of stress-strain type. Precisely, we can state the following results which in case $\zeta=2$ have already been established in \cite{FT}, Theorem 1.3.
\begin{Satz}\label{Eindeutigkeit}
Let us assume the hypotheses of Theorem \ref{dual}. Then the dual problem
\[
 R\rightarrow\text{max}\quad\text{in}\,\,L^{\infty}(\Omega)^{nM}
\]
admits a unique solution $\sigma$. We further have the duality formula 
\[
 \sigma=DF(\n^a u)\quad\text{a.e. on}\,\,\Omega,
\]
where $u$ denotes an arbitrary $K$-minimizer from the space $BV(\Omega)^M\cap L^{\zeta}(\Omega-D)^M$.
\end{Satz}

At this stage we finish the introduction by establishing a partial regularity result for any $K$-minimizer. For this purpose we assume that the density $F:\R^{nM}\rightarrow[0,\infty)$ satisfies the following (stronger) hypotheses
\begin{align}
 \label{v1s} &F\in C^2(\R^{nM}), F(0)=0, DF(0)=0,\\
\label{v2s} &|DF(P)|\leq\nu_1,\\
\label{v3s} &F(P)\geq \nu_2|P|-\nu_3,\\
\label{v4} &0<D^2F(P)(Q,Q)\leq \nu_4 (1+|P|)^{-1}|Q|^2.
\end{align}
with constants $\nu_1,\nu_2,\nu_4>0$, $\nu_3\in\R$, for all $P,Q\in\R^{nM}, Q\neq0$. Note that \eqref{v4} implies that $F$ is strictly convex and that of course, $F$ is of linear growth in the sense of \eqref{lg}.\\
Furthermore we suppose that $F$ is of the form
\begin{align}
\label{bs} F(P)=\Phi(|P|),\quad P\in\R^{nM},
\end{align}
with $\Phi:[0,\infty)\rightarrow[0,\infty)$ of class $C^2$. In order to have \eqref{v1s}--\eqref{v4} we then require that $\Phi$ fulfills the hypotheses $(1.3^*)-(1.4^*)$ and $(1.6^*)$ from \cite{BF2}.\\
As a consequence, Theorem \ref{BV} can be proven completely along the lines of the proof of Theorem 1.2 in \cite{BF2} without referring to the density results stated in Section 2 of \cite{FT} by observing that $K$-minimizing sequences $(u_m)$ can be chosen in such a way that we have
\[
 \sup\limits_{\Omega}{|u_m|}\leq\sup\limits_{\Omega-D}{|f|}
\]
which gives compactness of $(u_m)$ in $BV(\Omega)^M$ (we sketch the proof of this inequality in Section 5 of the present paper). Using \eqref{bs}, the functional $K$ from \eqref{K} reads 
\begin{align*}
 K[w]=&\int\limits_{\Omega}{\Phi(|\n^a w|)dx}+\overline{c}|\n^s w|(\Omega)+\frac{\lambda}{\zeta}\int\limits_{\Omega-D}{|w-f|^{\zeta}dx},\\
&w\in BV(\Omega)^M\cap L^{\zeta}(\Omega-D)^M,
\end{align*}
where the limit $\overline{c}:=\lim\limits_{t\rightarrow\infty}{\frac{\Phi(t)}{t}}$ exists in $(0,\infty)$ since $\Phi$ is of linear growth.\\
In the same spirit, the arguments used during the proof of Theorem 1.4 in \cite{BF2} now directly imply Theorem \ref{dual}. Moreover, we have uniqueness of the dual solution.\\
An example of a density satisfying the condition \eqref{bs} with $\Phi$ of class $C^2$ fulfilling $(1.3^*)-(1.4^*)$ and $(1.6^*)$ from \cite{BF2} is
\begin{align}
\label{Phimu} \Phi(t):=\Phi_{\mu}(t):=\int\limits_{0}^{t}{\int\limits_{0}^{s}{(1+r)^{-\mu}\,dr\,ds}},\quad t\geq0.
\end{align}
Note that the above example even satisfies the much stronger condition of $\mu$-ellipticity $(1.4_{\mu}^*)$ from \cite{BF2} with the prescribed elliptic parameter $\mu>1$.\\
Besides, we suppose that 
\begin{align}
 \label{funendlich} f\in L^{\infty}(\Omega-D)^M.
\end{align}
At this point we can state a partial regularity result for solutions of problem \eqref{Kv}. Precisely we can show
\begin{Satz}\label{partialregularity}
Let us fix a real number $\zeta>1$ and suppose that we have \eqref{L} as well as \eqref{funendlich}. Further we assume that $F$ satisfies \eqref{bs} with a function $\Phi$ of class $C^2$ fulfilling $(1.3^*)-(1.4^*)$ and $(1.6^*)$ from \cite{BF2}. Then for each $K$-minimizer $u$ there exists an open subset $\Omega_0^u$ of $\Omega$ such that $u\in C^{1,\gamma}(\Omega_0^u)^M$ for all $0<\gamma\leq\frac{1}{2}$ and $\mathcal{L}^{n}(\Omega-\Omega_0^u)=0$.
\end{Satz}
As outlined in, e.g. \cite{BF2}, the $\mu$-elliptic density $\Phi_{\mu}(t)$ acts as a very good candidate in order to approximate the $TV$-density $|P|$, $P\in\R^{nM}$. Particularly, the statement of Theorem \ref{partialregularity} clearly extends to the special integrand $F(P)=\Phi_{\mu}(|P|)$, i.e. it holds
\begin{Cor}\label{Corpartial}
Let us fix real numbers $\zeta,\mu>1$ and let us assume the validity of \eqref{L} as well as \eqref{funendlich}. Then, for each minimizer $u\in BV(\Omega)^M\cap L^{\zeta}(\Omega-D)^M$ of the functional
\begin{align*}
 \int\limits_{\Omega}{\Phi_{\mu}(|\n w|)}+\frac{\lambda}{\zeta}\int\limits_{\Omega-D}{|w-f|^{\zeta}dx},
\end{align*}
there exists an open subset $\Omega_0^u$ of $\Omega$ such that $u\in C^{1,\gamma}(\Omega_0^u)^M$ for all $0<\gamma\leq\frac{1}{2}$ and $\mathcal{L}^{n}(\Omega-\Omega_0^u)=0$.
\end{Cor}

Now let us give some comments on the partial regularity statement from Theorem \ref{partialregularity} and Corollary \ref{Corpartial}, respectively:
\begin{Rem}
\begin{itemize}
 \item As we will see in the proof of Theorem \ref{partialregularity} in Section 5 we apply Corollary 3.3 in \cite{S} in order to derive almost everywhere $C^{1,\gamma}$-regularity of each generalized minimizer. In this context, the limit $\widetilde{\gamma}=\frac{1}{2}$ serves as an optimal choice where we refer to Section 5 for more details. 
\item Considering the case $\mathcal{L}^n(D)>0$ and assuming that $\text{Int}(D)\neq\emptyset$ ($\text{Int}(D)$ denoting the set of interior points of $D$) as well as that $F$ satisfies \eqref{v1s}--\eqref{v4}, we can show that for each $K$-minimizer $u$ there exists an open subset $G^u$ of $G:=\text{Int}(D)$ such that we have $u\in C^{1,\alpha}(G^u)$ for any $\alpha\in(0,1)$ and $\mathcal{L}^n(G-G^u)=0$. This statement is an immediate consequence of Theorem 1.1 in \cite{AG} and it is possible to drop the structure condition \eqref{bs} on $F$ in this situation.
\item Of course, the statement of Theorem \ref{partialregularity} remains valid if we replace \eqref{v4} through the much stronger condition of $\mu$-ellipticity
\begin{align*}
 \nu_5 (1+|P|)^{-\mu}|Q|^2\leq D^2F(P)(Q,Q)\leq \nu_4 (1+|P|)^{-1}|Q|^2
\end{align*}
with $\nu_5>0$, exponent $\mu>1$ and for all $P,Q\in\R^{nM}$. In this case, the function $\Phi$ from \eqref{bs} shall satisfy $(1.4^{*}_{\mu})$ from \cite{BF2} with the prescribed parameter $\mu>1$. Moreover we like to emphasize that partial regularity of any $K$-minimizer $u$ as stated in Theorem \ref{partialregularity} does not depend on the choice of the parameter $\mu>1$.
\end{itemize}
\end{Rem}

\begin{Rem}
 Following the lines of the proof of Theorem \ref{partialregularity} we can see that the statement extends to the case $\zeta=1$. Furthermore we can even drop the structure condition \eqref{bs} for our density $F$ and the $L^{\infty}$-assumption on our partial observation $f$ since the Morrey assumption \eqref{mo} is void in this case.
\end{Rem}

As a consequence of the previous statements the question arises if it is possible to establish full $C^{1,\alpha}$-interior regularity of $K$-minimizers under the particular assumptions mentioned before. Considering the case $\zeta=2$ and a density $F$ being $\mu$-elliptic for some $\mu\in(1,2)$, a satisfying answer is given for the scalar case $(M=1)$ in two dimensions $(n=2)$: in this situation it is shown that we can solve problem \eqref{vp1} in the Sobolev space $W^{1,1}(\Omega)$ without considering a suitable relaxed version in the space $BV(\Omega)$ and that the corresponding solution $u$ is unique (see \cite{BF1}, Theorem 1.3). Further, in \cite{BFT}, it was proved that $u\in C^{1,\alpha}(\Omega)$ for all $\alpha\in(0,1)$.\\
Moreover, in the forthcoming paper \cite{T}, it is shown that these regularity results extend to any dimension $n$ and arbitrary codimension $M$.\\
If we fix any number $\zeta>1, \zeta\neq2$, in the functional $I$ from \eqref{vp1}, it lies completely in the dark whether it is possible to show full $C^{1,\alpha}$-interior regularity for arbitrary $K$-minimizers.\\
\newline
Our paper is organized as follows: in Section 2 we state and prove a density result for functions of $BV$-type. In Section 3 we investigate the dual problem while in Section 4 we prove uniqueness of the dual solution and will establish the duality formula in addition. Finally, in Section 5 we study regularity properties of generalized minimizer. 

\end{section}

\begin{section}{A density result for $BV$-functions}

In this section we provide a density result for functions of $BV$-type. 

\begin{Lem}\label{BV_Approx}
Let $\Omega\subset\R^n$ denote a bounded Lipschitz domain and consider a measurable subset $D$ of $\Omega$ such that $\mathcal{L}^n(D)<\mathcal{L}^n(\Omega)$. Consider $u\in BV(\Omega)^M\cap L^q(\Omega-D)^M$ for some $q\in(\frac{n}{n-1},\infty)$. Then there exists a sequence $(u_k)\subset C^{\infty}(\overline{\Omega})^M$ such that (as $k\rightarrow\infty$)
\begin{enumerate}[(i)]
 \item $u_k\rightarrow u\,\,\text{in}\,\,L^1(\Omega)^M$,
\item $u_k\rightarrow u\,\,\text{in}\,\,L^q(\Omega-D)^M$,
\item $\displaystyle{\int\limits_{\Omega}{\sqrt{1+|\nabla u_k|^2}dx}\rightarrow\int\limits_{\Omega}{\sqrt{1+|\nabla u|^2}}}$.
\end{enumerate}
\end{Lem}

\begin{Rem}
A proof of Lemma \ref{BV_Approx} has already been established in \cite{FT}, Lemma 2.2, where in place of (iii) it was proven the validity of ($|\n u|(\Omega)$ denoting the total variation of the tensor-valued signed Radon measure $\n u$ on $\Omega$)
\begin{align*}
 (iii)^*\quad \int\limits_{\Omega}{|\n u_k|dx}\rightarrow\int\limits_{\Omega}{|\n u|}=|\n u|(\Omega)\quad\text{as}\,\,k\rightarrow\infty.
\end{align*}
Actually we need (iii) in order to apply the continuity theorem of Reshetnyak (see, e.g., \cite{AG},  Proposition 2.2) which implies continuity of the functional ($w\in BV(\Omega)^M$)
\[
 \widetilde{K}[w]=\int\limits_{\Omega}{F(\nabla^a w)dx}+\int\limits_{\Omega}{F^{\infty}\bigg(\frac{\nabla ^s w}{|\nabla ^s w|}\bigg)d|\nabla^s w|}
\]
w.r.t. the convergences stated in (i) and (iii). For that reason we follow the hint of \cite{FT} (see Remark 2.3 in this reference) and adopt the procedure that has been used therein (see \cite{FT}, proof of Lemma 2.2) in order to verify that the construced sequence $(u_m)\subset C^{\infty}(\overline{\Omega})$ in fact satisfies (iii).
\end{Rem}

\begin{proof}[Proof of Lemma \ref{BV_Approx}]
Let us choose a smooth bounded domain $\widetilde{\Omega}$ such that $\Omega\Subset\widetilde{\Omega}$. Given a function $w\in BV(\widetilde{\Omega})^M$ we first recall that in agreement with the general concept of applying a convex function to a measure (see, e.g., \cite{DT}) the quantity $\int\limits_{\Omega}{\sqrt{1+|\n w|^2}}$ is defined as follows:
\[
 \sqrt{1+|\n w|^2}(B):=\int\limits_{B}{\sqrt{1+|\n w|^2}}:=\int\limits_{B}{\sqrt{1+|\n^a w|^2}dx}+|\n^s w|(B),\quad B\in\mathcal{B}(\widetilde{\Omega}).
\]
Now let $u_0\in L^1(\partial\Omega)^M$ denote the trace of the given function $u\in BV(\Omega)^M\cap L^q(\Omega-D)^M$ whose properties are summarized in e.g. \cite{AFP}, Theorem 3.87, p.180/181. With $\widetilde{\Omega}$ as above we let $u_0:=0$ on $\partial\widetilde{\Omega}$, thus $u_0\in L^1(\partial G)^M$ where $G:=\widetilde{\Omega}-\overline{\Omega}$. Referring to \cite{Giu}, Theorem 2.16, p.39, we can find $v\in W^{1,1}(G)^M$ having trace $u_0$ on $\partial G$ and such that
\begin{align}
\label{bv1} \|v\|_{W^{1,1}(G)}\leq c\|u_0\|_{L^1(\partial G)} 
\end{align}
with $c$ depending on $\partial G$ but independent of $u_0$ and $v$. We then let 
\begin{align*} 
 \widetilde{u}:=
\left\{
  \begin{array}{ll}
   u, &\quad\text{on}\,\,\Omega \\
   v, &\quad\text{on}\,\,\widetilde{\Omega}-\overline{\Omega}
  \end{array}
\right.
\end{align*}
and observe $\widetilde{u}\in BV(\widetilde{\Omega})^M$, which follows from \cite{AFP}, Corollary 3.89, p.183, and the fact that \eqref{bv1} implies $v\in BV(G)^M$. Viewing $\n u$ (resp. $\n v$) as measures on $\widetilde{\Omega}$ concentrated on $\Omega$ (resp. $\widetilde{\Omega}-\overline{\Omega}$) and recalling the definition of $v$ we further deduce from the above reference the identity
\begin{align}
 \label{bv2} \n\widetilde{u}=\n u+\n v
\end{align}
as measures on $\widetilde{\Omega}$. For $m\in\N$ we now let $\Phi_m:\R^M\rightarrow\R^M$ given by
\begin{align*}
  \Phi_m(y) :=
\left\{
  \begin{array}{ll}
    y, & |y|\leq m \\
    m\frac{y}{|y|}, & |y|\geq m
  \end{array}
\right.
\end{align*}
and observe for the sequence $\widetilde{u}_m:=\Phi_m\circ\widetilde{u}$ (compare the first part of the proof of Theorem 3.96 on p.189 in \cite{AFP})
\begin{align}
 \label{bv3} \widetilde{u}_m\in BV(\widetilde{\Omega})^M,\quad |\n\widetilde{u}_m|\leq\text{Lip}(\Phi_m)|\n\widetilde{u}|=|\n\widetilde{u}|,
\end{align}
which means $|\n\widetilde{u}_m|(E)\leq|\n\widetilde{u}|(E)$ for any Borel set $E\subset\widetilde{\Omega}$. In particular, from $|\n\widetilde{u}|(\partial\Omega)=0$ (recall \eqref{bv2}) it follows that
\begin{align}
 \label{bv4} |\n\widetilde{u}_m|(\partial\Omega)=0,\quad m\in\N.
\end{align}
As a consequence of \eqref{bv3} we further obtain (compare \cite{AG}, Proposition 2.1)
\begin{align}
\label{bv41} \sqrt{1+|\n\widetilde{u}_m|^2}(E)\leq \sqrt{1+|\n\widetilde{u}|^2}(E) 
\end{align}
for any Borel set $E\subset\widetilde{\Omega}$ while on the other hand we get by using \eqref{bv4}
\begin{align}
 \label{bv42} \sqrt{1+|\n\widetilde{u}_m|^2}(\partial\Omega)=0,\quad m\in\N.
\end{align}
Applying Lebesgue's theorem on dominated convergence it follows (as $m\rightarrow\infty$)
\begin{align}
 \label{bv5} &\widetilde{u}_m\rightarrow\widetilde{u}\quad\text{in}\,\, L^1(\widetilde{\Omega})^M,\\
\label{bv6} &\widetilde{u}_m\rightarrow u\quad\text{in}\,\, L^q(\Omega-D)^M,
\end{align}
and \eqref{bv5} combined with lower semicontinuity implies
\begin{align*}
\sqrt{1+|\n\widetilde{u}|^2}(\widetilde{\Omega})\leq\liminf\limits_{m\rightarrow\infty}{\sqrt{1+|\n\widetilde{u}_m|^2}(\widetilde{\Omega})}
\end{align*}
From \eqref{bv41} we get
\begin{align*}
\sqrt{1+|\n\widetilde{u}_m|^2}(\widetilde{\Omega})\leq \sqrt{1+|\n\widetilde{u}|^2}(\widetilde{\Omega}) 
\end{align*}
thus
\begin{align}
 \label{bv7} \sqrt{1+|\n\widetilde{u}_m|^2}(\widetilde{\Omega})\rightarrow \sqrt{1+|\n\widetilde{u}|^2}(\widetilde{\Omega}) ,\quad m\rightarrow\infty.
\end{align}
Clearly we can replace $\widetilde{\Omega}$ in \eqref{bv7} by the domain $\Omega$, so that in combination with \eqref{bv5} and \eqref{bv6} it holds for a subsequence (recall that by \eqref{bv2} $|\n\widetilde{u}|(\Omega)=|\n u|(\Omega)$)
\begin{align}
 \label{bv8} \begin{split} 
&\|\widetilde{u}_{m_k}-u\|_{L^1(\Omega)}+\|\widetilde{u}_{m_k}-u\|_{L^q(\Omega-D)}\\
&+\bigg|\sqrt{1+|\n\widetilde{u}_{m_k}|^2}(\Omega)-\sqrt{1+|\n u|^2}(\Omega)\bigg|\leq\frac{1}{k},\quad k\in\N.
\end{split}
\end{align}
In a next step we consider a radius $\rho>0$ (being sufficently small) and consider the mollification $(\widetilde{u}_{m_k})_{\rho}=\widetilde{u}_{m_k}\ast\eta_{\rho}$ where $\eta$ denotes a smoothing kernel. By quoting well-known convergence properties of the mollification (see, e.g., \cite{Ad}, Lemma 2.18, p.29/30) we may directly infer (note that we have the convergence $(\widetilde{u}_{m_k})_{\rho}\rightarrow\widetilde{u}_{m_k}$ in $L^p_{\loc}(\widetilde{\Omega})^M$ as $\rho\downarrow0$ for any $p\in[1,\infty)$)
\begin{align}
 \label{bv9} \|(\widetilde{u}_{m_k})_{\rho}-\widetilde{u}_{m_k}\|_{L^1(\Omega)}+\|(\widetilde{u}_{m_k})_{\rho}-\widetilde{u}_{m_k}\|_{L^q(\Omega-D)}\rightarrow0\quad\text{as}\,\,\rho\downarrow0.
\end{align}
In order to show the required convergence for the modification of the total variation we adopt ideas as applied in \cite{Giu}, Proposition 1.15, p.12 and adapt the procedure as carried out in \cite{Bi1}, Proof of Lemma B.1, p.185-188, i.e. we consider the function 
\begin{align*}
  g(Z)=\sqrt{1+|Z|^2}-1,\quad Z\in\R^{nM}
\end{align*}
and notice that its conjugate function $g^*$ is given by
\begin{align*}
 g^*(Q)=
\left\{
  \begin{array}{ll}
   +\infty, & \text{if}\,\,|Q|>1,\\
  1-\sqrt{1-|Q|^2}, & \text{if}\,\,|Q|\leq1.
  \end{array}
\right.
\end{align*}
for any $Q\in\R^{nM}$. Further, we note that $g^*$ is convex satisfying $g^*(0)=0$ as well as $g^*\geq0$.\\
Following \cite{DT}, Definition 1.2,  we can state the following representation formula for the measure $\int\limits_{\widetilde{\Omega}}{g(\n w)}$ with $w\in BV(\widetilde{\Omega})^M$ (compare also (3) in \cite{Bi1} on p.186)
\begin{align*}
 \int\limits_{\widetilde{\Omega}}{g(\n w)}=\sup\limits_{\varkappa\in C^{\infty}_0(\widetilde{\Omega})^{nM},|\varkappa|\leq1}\bigg[-\int\limits_{\widetilde{\Omega}}{w\,\text{div}\varkappa\,dx}-\int\limits_{\widetilde{\Omega}}{g^*(\varkappa)dx}\bigg].
\end{align*}
We like to remark that the above representation formula extends to any Borel set $E\subset\widetilde{\Omega}$ by letting (compare \cite{DT},Definition 1.2 again)
\[
\int\limits_{E}{g(\n w)}=\sup\limits_{\varkappa\in C^{\infty}_0(\widetilde{\Omega})^{nM},|\varkappa|\leq1}\bigg[\int\limits_{E}{\varkappa:\n w}-\int\limits_{\widetilde{\Omega}}{g^*(\varkappa)dx}\bigg].
\]
By lower semicontinuity (recall \eqref{bv9}) we have
\begin{align}
 \label{bv10} \int\limits_{\Omega}{g(\n\widetilde{u}_{m_k})}\leq\liminf\limits_{\rho\downarrow0}{\int\limits_{\Omega}{g(\n(\widetilde{u}_{m_k})_{\rho})}}.
\end{align}
For verifying the reverse inequality we fix $\varkappa\in C^{\infty}_0(\Omega)^{nM}$ satisfying $|\varkappa|\leq1$ and get (we identify $\varkappa$ with its zero-extension to $\R^n$)
\begin{align}
\begin{split}\label{bv11}
&-\int\limits_{\widetilde{\Omega}}{(\widetilde{u}_{m_k})_{\rho}\,\text{div}\varkappa\,dx}-\int\limits_{\widetilde{\Omega}}{g^*(\varkappa)dx}\\
&=-\int\limits_{\widetilde{\Omega}}{\widetilde{u}_{m_k}\,\text{div}(\varkappa)_{\rho}\,dx}-\int\limits_{\widetilde{\Omega}}{g^*(\varkappa)dx}.
\end{split}
\end{align}
Besides, we obtain (recall $\varkappa\equiv0$ on $\R^n-\Omega$ and $g^*(0)=0$)
\begin{align}
\begin{split}\label{bv12}
 -\int\limits_{\widetilde{\Omega}}{g^*(\varkappa)dx}&=-\int\limits_{\R^n}{g^*(\varkappa)dx}\\
&=-\int\limits_{\R^n}{g^*((\varkappa)_{\rho})dx}+\bigg\{\int\limits_{\R^n}{[g^*((\varkappa)_{\rho})-g^*(\varkappa)]dx}\bigg\}.
\end{split}
\end{align}
In order to handle the second integral on the r.h.s. of \eqref{bv12} we use Jensen's inequality that gives
\begin{align}
 \label{bv13} g^*((\varkappa)_{\rho})\leq g^*(\varkappa)_{\rho}.
\end{align}
By means of \eqref{bv13} and due to Fubini's theorem we may derive 
\begin{align}
\label{bv14} \int\limits_{\R^n}{[g^*((\varkappa)_{\rho})-g^*(\varkappa)]dx}\leq 0.
\end{align}
As a result of \eqref{bv12}--\eqref{bv14}, \eqref{bv11} turns into 
\begin{align}
\begin{split}\label{bv15}
&-\int\limits_{\widetilde{\Omega}}{(\widetilde{u}_{m_k})_{\rho}\,\text{div}\varkappa\,dx}-\int\limits_{\widetilde{\Omega}}{g^*(\varkappa)dx}\\
&\leq -\int\limits_{\widetilde{\Omega}}{\widetilde{u}_{m_k}\,\text{div}(\varkappa)_{\rho}\,dx}-\int\limits_{\widetilde{\Omega}}{g^*((\varkappa)_{\rho})dx}.
\end{split}
\end{align}
Taking standard properties of the mollification into account (see, e.g.,\cite{Giu}, p.11) we get $|(\varkappa)_{\rho}|\leq1$ since $|\varkappa|\leq1$  as well as $\text{spt}\,(\varkappa)_{\rho}\subset\Omega_{\rho}:=\{x;\text{dist}(x,\Omega)\leq\rho\}$ since $\text{spt}\,\varkappa\subset\Omega$. Consequently, \eqref{bv15} turns into
\begin{align*}
-\int\limits_{\widetilde{\Omega}}{(\widetilde{u}_{m_k})_{\rho}\,\text{div}\varkappa\,dx}-\int\limits_{\widetilde{\Omega}}{g^*(\varkappa)dx}\leq \int\limits_{\Omega_{\rho}}{g(\n\widetilde{u}_{m_k})}.
\end{align*}
At this point we take the supremum over all such $\varkappa$ and this results in
\begin{align*}
\int\limits_{\Omega}{g(\n(\widetilde{u}_{m_k})_{\rho})}\leq  \int\limits_{\Omega_{\rho}}{g(\n\widetilde{u}_{m_k})}.
\end{align*}
Thus, we obtain
\begin{align*}
 \limsup\limits_{\rho\downarrow0}{\int\limits_{\Omega}{g(\n(\widetilde{u}_{m_k})_{\rho})}}\leq \lim\limits_{\rho\downarrow0}{\int\limits_{\Omega_{\rho}}{g(\n\widetilde{u}_{m_k})}}=\int\limits_{\overline{\Omega}}{g(\n\widetilde{u}_{m_k})}.
\end{align*}
On account of \eqref{bv42} we may conclude
\begin{align*}
 \int\limits_{\partial\Omega}{g(\n\widetilde{u}_{m_k})}=0
\end{align*}
that implies
\begin{align}
 \label{bv16} \limsup\limits_{\rho\downarrow0}{\int\limits_{\Omega}{g(\n(\widetilde{u}_{m_k})_{\rho})}}\leq \int\limits_{\Omega}{g(\n\widetilde{u}_{m_k})}
\end{align}
Combining \eqref{bv16} with \eqref{bv10} we have
\begin{align} 
 \lim\limits_{\rho\downarrow0}{\int\limits_{\Omega}{g(\n(\widetilde{u}_{m_k})_{\rho})}}=\int\limits_{\Omega}{g(\n\widetilde{u}_{m_k})}
\end{align}
which clearly implies
\begin{align*}
 \lim\limits_{\rho\downarrow0}{\sqrt{1+|\n(\widetilde{u}_{m_k})_{\rho}|^2}(\Omega)}=\sqrt{1+|\n\widetilde{u}_{m_k}|^2}(\Omega).
\end{align*}
Passing now to a suitable subsequence of radii $\rho_k$ ($k\in\N$) we can arrange
\begin{align}
 \begin{split}\label{bv17}
&\|\widetilde{u}_{m_k}-(\widetilde{u}_{m_k})_{\rho_k}\|_{L^1(\Omega)}+\|\widetilde{u}_{m_k}-(\widetilde{u}_{m_k})_{\rho_k}\|_{L^q(\Omega-D)}\\
&+\bigg|\sqrt{1+|\n\widetilde{u}_{m_k}|^2}(\Omega)-\sqrt{1+|\n(\widetilde{u}_{m_k})_{\rho_k}|^2}(\Omega)\bigg|\\
&\leq\frac{1}{k}.
\end{split}
\end{align}                                                                                                                                                                                                                                                                                                                                                                                                                                                                                                                                                        
Putting together \eqref{bv8} and \eqref{bv17} we see that the sequence $u_k:=(\widetilde{u}_{m_k})_{\rho_k}$ is of class $C^{\infty}(\overline{\Omega})$ and has the desired properties. 
\end{proof}

\end{section}

\begin{section}{Dual solutions. Proof of Theorem \ref{dual}}
Let us assume the validity of the hypotheses of Theorem \ref{dual} and note that a proof of assertion (a) can probably be deduced by quoting \cite{FS}, Theorem 1.2.1, p.15/16 or \cite{ET}, Proposition 2.3, Chapter III, p.52.  As in \cite{BF2}, proof of Theorem 1.4 or \cite{FT}, proof of Theorem 1.2, respectively, we prefer to prove Theorem \ref{dual} in a more constructive way based on an approximation of our original variational problem through a sequence of more regular problems having smooth solutions with appropriate convergence properties. In particular, this sequence might be of interest for numerical computations.\\
To be more precise, for fixed $\delta\in(0,1]$ we look at the following problem
\begin{align}
\label{Id}\begin{split}  I_{\delta}[w]:=&\intop_{\Omega}F_{\delta}(\nabla w)\mbox{ d}x+\frac{\lambda}{\zeta}\intop_{\Omega-D}|w-f|^{\zeta}dx\rightarrow\text{min}\\
&\text{in}\,\,W^{1,2}(\Omega)^M\cap L^{\zeta}(\Omega-D)^M
\end{split}
\end{align}
where 
\begin{align*}
 F_{\delta}(Z):=\frac{\delta}{2}|Z|^{2}+F(Z),\quad Z\in\R^{nM}.
\end{align*}
Note that by the continuity of Sobolev's embedding $W^{1,2}(\Omega)^M\hookrightarrow L^{\frac{2n}{n-2}}(\Omega)^M$ (see, e.g., \cite{Ad}, Theorem 5.4, p.97/98) the requirement \grqq $u\in L^{\zeta}(\Omega-D)^M$\grqq\,\,acts as an additional constraint if $\zeta>\frac{2n}{n-2}$. Considering the special case $n=2$ with arbitrary codimension $M\geq1$, Sobolev's embedding theorem directly implies that we can discuss problem \eqref{Id} in the whole space $W^{1,2}(\Omega)^M$ for any choice of $\zeta>1$.\\
As a matter of fact, problem \eqref{Id} has at most one solution $\ud\in W^{1,2}(\Omega)^M\cap L^{\zeta}(\Omega-D)^M$: assuming that $u_1, u_2$ are solutions problem \eqref{Id} it follows $\n u_1=\n u_2$ on $\Omega$ and as a consequence it holds $u_1=u_2$ by virtue of \eqref{L}.\\
Next, with $\delta\in(0,1]$ being fixed, we denote by $(u_m)$ an $I_{\delta}$-minimizing sequence from $W^{1,2}(\Omega)^M\cap L^{\zeta}(\Omega-D)^M$ for which it holds
\begin{align*}
 &\sup\limits_{m}{\|\n u_m\|_{L^2(\Omega)}}\leq c(\delta)<\infty,\\
&\sup\limits_{m}{\|\n u_m\|_{L^1(\Omega)}}<\infty,\\
&\sup\limits_{m}{\|u_m-f|\|_{L^{\zeta}(\Omega-D)}}<\infty
\end{align*}
where the linear growth of $F$ has been used.\\
As stated on p.380 of \cite{AFP} we have the following variant of Poincar\'{e}'s inequality (recall $\mathcal{L}^n(\Omega-D)>0$ by \eqref{L})
\begin{align}
\label{poincare} \|v-(v)_{\Omega-D}\|_{L^q(\Omega)}\leq c\|\n v\|_{L^q(\Omega)}
\end{align}
valid for $v\in W^{1,q}(\Omega)^M$, $1\leq q<\infty$, where $c$ denotes a positive constant independent of $v$. It is to mention that by standard approximation (see, e.g., \cite{AFP}, Theorem 3.9, p.122) the above inequality \eqref{poincare} remains valid for functions $v\in BV(\Omega)^M$. The quadratic variant of inequality \eqref{poincare} directly implies
\begin{align*}
 \sup\limits_{m}{\|u_m\|_{L^2(\Omega)}}<\infty,
\end{align*}
thus, we get $u_m\rightharpoondown:\ud$ in $W^{1,2}(\Omega)^M$ at least for a suitable subsequence. Besides we obtain $u_m\rightharpoondown\ud$ in $L^{\zeta}(\Omega-D)^M$ after passing to another subsequence and the weak $L^{\zeta}(\Omega-D)^M$-convergence yields
\begin{align}
\label{uhstlp1} \int\limits_{\Omega-D}{|\ud-f|^{\zeta}dx}\leq\liminf\limits_{m\rightarrow\infty}{\int\limits_{\Omega-D}{|u_m-f|^{\zeta}dx}},
\end{align}
i.e. $\ud\in W^{1,2}(\Omega)^M\cap L^{\zeta}(\Omega-D)^M$ and $I_{\delta}[\ud]$ is well-defined.\\
Quoting standard results on lower semicontinuity we have
\begin{align}
 \label{uhstw121} \int\limits_{\Omega}{F_{\delta}(\n\ud)dx}\leq \liminf\limits_{m\rightarrow\infty}{\int\limits_{\Omega}{F_{\delta}(\n u_m)dx}}
\end{align}
and combining \eqref{uhstlp1} with \eqref{uhstw121} it follows
\begin{align*}
 I_{\delta}[\ud]&\leq\liminf\limits_{m\rightarrow\infty}{\int\limits_{\Omega}{F_{\delta}(\n u_m)dx}}+\liminf\limits_{m\rightarrow\infty}{\frac{\lambda}{\zeta}\int\limits_{\Omega-D}{|u_m-f|^{\zeta}dx}}\\
&\leq\liminf\limits_{m\rightarrow\infty}{I_{\delta}[u_m]}=\inf\limits_{W^{1,2}(\Omega)^M\cap L^{\zeta}(\Omega-D)^M}{I_{\delta}}
\end{align*}
implying that $\ud$ solves problem \eqref{Id}.\\
Observing that we have the uniform estimate $I_{\delta}[\ud]\leq I_{\delta}[0]=I[0]$ and using the linear growth of $F$ in addition we may derive
\begin{align}
\label{uniformdelta} \begin{split} 
&\sup\limits_{\delta}{\|\n\ud\|_{L^1(\Omega)}}<\infty,\\
&\sup\limits_{\delta}{\|\ud-f\|_{L^{\zeta}(\Omega-D)}}<\infty,\\
&\sup\limits_{\delta}{\,\,\delta\int\limits_{\Omega}{|\n\ud|^2dx}}<\infty.                 
\end{split}
\end{align}
Applying Poincar\'{e}'s inequality \eqref{poincare} (choose $q=1$) we additionally get
\begin{align}
 \label{L1uniform} \sup\limits_{\delta}{\|\ud\|_{L^1(\Omega)}}<\infty
\end{align}
and combining \eqref{uniformdelta} and \eqref{L1uniform} it therefore follows
\begin{align}
 \label{w11uniform}  \sup\limits_{\delta}{\|\ud\|_{W^{1,1}(\Omega)}}<\infty.
\end{align}
In accordance with \eqref{w11uniform} and \eqref{uniformdelta} it holds (at least for an appropriate subsequence $\delta\downarrow0$)
\begin{align*}
&\ud\rightarrow:\overline{u}\quad\text{in}\,\,L^1(\Omega)^M\,\,\text{and a.e.},\\
&\ud\rightharpoondown\overline{u}\quad\text{in}\,\,L^{\zeta}(\Omega-D)^M
\end{align*}
for a function $\overline{u}\in BV(\Omega)^M$ (compare, e.g., \cite{AFP}, Theorem 3.23, p.132). In addition the weak $L^{\zeta}(\Omega-D)^M$ convergence yields
\begin{align*}
 \int\limits_{\Omega-D}{|\overline{u}-f|^{\zeta}dx}\leq\liminf\limits_{\delta\downarrow0}{\int\limits_{\Omega-D}{|\ud-f|^{\zeta}dx}},
\end{align*}
i.e. our limit function $\overline{u}$ belongs to the space $BV(\Omega)^M\cap L^{\zeta}(\Omega-D)^M$.\\
Next we set
\begin{align}
 \label{dv2} \tau_{\delta}:=DF(\n\ud)\quad\text{and}\quad\sigma_{\delta}:=DF_{\delta}(\n\ud)=\delta\n\ud+\tau_{\delta}
\end{align}
and note that \eqref{uniformdelta} gives
\begin{equation}
\label{dv3} \delta\nabla u_{\delta}\rightarrow0\mbox{ in }L^{2}(\Omega)^{nM}\mbox{ as }\delta\rightarrow0.
\end{equation}
Moreover, \eqref{v2} shows that $\tau_{\delta}$ is uniformly bounded w.r.t. $\delta$, i.e.
\begin{align}
 \label{dv4} \sup\limits_{\delta}{||\tau_{\delta}||_{L^{\infty}(\Omega)}}<\infty.
\end{align}
So, by fixing a suitable sequence $\delta\downarrow0$, we may derive
\begin{equation}
\label{dv5} \sigma_{\delta}\rightharpoondown:\sigma\mbox{ in }L^{2}(\Omega)^{nM}\mbox{ and }\tau_{\delta}\overset{*}{\rightharpoondown}:\tau\mbox{ in }L^{\infty}(\Omega)^{nM}.
\end{equation}
and by combining \eqref{dv3} with \eqref{dv5} it follows $\sigma=\tau$ on $\Omega$.\\
As in \cite{BF2}, proof of Theorem 1.4, or \cite{FT}, proof of Theorem 1.2, respectively, it will turn out that $\sigma$ is a solution of the dual problem. To justify this assertion we state that by the $I_{\delta}$-minimality of $\ud$ we get for all $\varphi\in W^{1,2}(\Omega)^M\cap L^{\zeta}(\Omega-D)^M$
\begin{align}
 \label{euler}
\begin{split} 0&=\frac{d}{dt|_0}I_{\delta}[\ud+t\varphi]\\
&=\int\limits_{\Omega}{\tau_{\delta}:\n\varphi dx}+\delta\int\limits_{\Omega}{\n u_{\delta}:\n\varphi dx}+\lambda\int\limits_{\Omega-D}{|u_{\delta}-f|^{\zeta-2}(u_{\delta}-f)\cdot\varphi dx}.
\end{split}
\end{align}
An application of the identity $F(\nabla u_{\delta})=\tau_{\delta}:\nabla u_{\delta}-F^*(\tau_{\delta})$ (see, e.g., \cite{ET}, Proposition 5.1, p.21) then leads to
\begin{align*}
I_{\delta}[u_{\delta}]=\frac{\delta}{2}\intop_{\Omega}|\nabla u_{\delta}|^{2}dx+\intop_{\Omega}[\tau_{\delta}:\nabla u_{\delta}-F^{*}(\tau_{\delta})]dx+\frac{\lambda}{\zeta}\intop_{\Omega-D}|u_{\delta}-f|^{\zeta}dx.
\end{align*}
Next we observe that $\varphi=u_{\delta}$ serves as an admissible choice in \eqref{euler} and it follows
\begin{equation}\label{dv6}
\begin{aligned}
I_{\delta}[\ud]=&-\frac{\delta}{2}\int\limits_{\Omega}{|\n\ud|^2dx}+\int\limits_{\Omega}{(-F^{*}(\tau_{\delta}))dx}+\frac{\lambda}{\zeta}\int\limits_{\Omega-D}{|\ud-f|^{\zeta}dx}\\
&-\lambda\int\limits_{\Omega-D}{|\ud-f|^{\zeta-2}(\ud-f)\cdot\ud\,dx}.
\end{aligned}
\end{equation}
Now we let $v\in W^{1,1}(\Omega)^M\cap L^{\zeta}(\Omega-D)^M$. Recalling \eqref{Ineu} and the definition of the dual functional $R$ we can state 
\begin{align} \label{dv7} \begin{split}
I[v]&=\sup\limits_{\varkappa\in L^{\infty}(\Omega)^{nM}}{l(v,\varkappa)}\\
&\geq l(v,\rho)\geq\inf\limits_{w\in W^{1,1}(\Omega)^M\cap L^{\zeta}(\Omega-D)^M}{l(w,\rho)}=R[\rho]
\end{split}
\end{align}
for any $\rho\in L^{\infty}(\Omega)^{nM}$, hence, we conclude
\begin{align*}
 \inf\limits_{W^{1,1}(\Omega)^M\cap L^{\zeta}(\Omega-D)^M}{I}\geq\sup\limits_{L^{\infty}(\Omega)^{nM}}{R}.
\end{align*}
Besides we observe $\inf\limits_{v\in W^{1,1}(\Omega)^M\cap L^{\zeta}(\Omega-D)^M}{I[v]}\leq I[u_{\delta}]\leq I_{\delta}[u_{\delta}]$ and by means of \eqref{dv6} as well as \eqref{dv7} we obtain
\begin{align}
\label{dv8}
\begin{split}
&\sup\limits_{\rho\in L^{\infty}(\Omega)^{nM}}{R[\rho]}\\
&\leq\inf\limits_{v\in W^{1,1}(\Omega)^M\cap L^{\zeta}(\Omega-D)^M}{I[v]}\\
&\leq -\frac{\delta}{2}\int\limits_{\Omega}{|\n\ud|^2 dx}+\int\limits_{\Omega}{(-F^*(\tau_{\delta}))dx}\\
&+\frac{\lambda}{\zeta}\int\limits_{\Omega-D}{|u_{\delta}-f|^{\zeta}dx}-\lambda\int\limits_{\Omega-D}{|\ud-f|^{\zeta-2}(\ud-f)\cdot\ud\,dx}\\
&\leq \int\limits_{\Omega}{(-F^*(\tau_{\delta}))dx}-\bigg(\lambda-\frac{\lambda}{\zeta}\bigg)\int\limits_{\Omega-D}{|\ud-f|^{\zeta}dx}\\
&-\lambda\int\limits_{\Omega-D}{|\ud-f|^{\zeta-2}(\ud-f)\cdot f\,dx}.
\end{split}
\end{align}
Next we want to pass to the limit $\delta\downarrow0$ in \eqref{dv8}. For that reason we initially state that $\int\limits_{\Omega}{(-F^*(\cdot))dx}$ is upper semicontinuous w.r.t. weak-$*$ convergence. Due to the weak $L^{\zeta}(\Omega-D)^M$-convergence we further get
\[
\int\limits_{\Omega-D}{|\overline{u}-f|^{\zeta}dx}\leq\liminf\limits_{\delta\downarrow0}{\int\limits_{\Omega-D}{|\ud-f|^{\zeta}dx}}.
\]
Considering now the last integral on the r.h.s. of \eqref{dv8} we first state that due to \eqref{uniformdelta} the function $v_{\delta}:=|\ud-f|^{\zeta-2}(\ud-f)$ is uniformly bounded in $L^{\frac{\zeta}{\zeta-1}}(\Omega-D)^M$ w.r.t. $\delta$, i.e. we have ($\zeta>1$) $v_{\delta}\rightharpoondown:v$ in $L^{\frac{\zeta}{\zeta-1}}(\Omega-D)^M$ as $\delta\rightarrow0$ at least for a subsequence. Recalling $\ud\rightarrow\overline{u}$ a.e. we also get $v_{\delta}\rightarrow|\overline{u}-f|^{\zeta-2}(\overline{u}-f)$ a.e. for a further subsequence. Thus we may derive $v=|\overline{u}-f|^{\zeta-2}(\overline{u}-f)$ a.e. and combining this fact with the weak $L^{\frac{\zeta}{\zeta-1}}(\Omega-D)^M$-convergence it finally follows (recall \eqref{f})
\[
 \int\limits_{\Omega-D}{v_{\delta}\cdot f\,dx}\longrightarrow\int\limits_{\Omega-D}{|\overline{u}-f|^{\zeta-2}(\overline{u}-f)\cdot f\,dx}\quad\text{as}\,\,\delta\rightarrow0.
\]
Passing now to the limit $\delta\rightarrow0$ in \eqref{dv8} we get by using the previous convergences (note that we have the appropriate signs in \eqref{dv8})
\begin{align}
\label{dv9}
\begin{split}
\sup\limits_{\rho\in L^{\infty}(\Omega)^{nM}}{R[\rho]}&\leq\inf\limits_{v\in W^{1,1}(\Omega)^M\cap L^{\zeta}(\Omega-D)^M}{I[v]}\\
&\leq \int\limits_{\Omega}{(-F^*(\tau))dx}-\bigg(\lambda-\frac{\lambda}{\zeta}\bigg)\int\limits_{\Omega-D}{|\overline{u}-f|^{\zeta}dx}\\
&-\lambda\int\limits_{\Omega-D}{|\overline{u}-f|^{\zeta-2}(\overline{u}-f)\cdot f\,dx}.
\end{split}
\end{align}
Next we pass to the limit $\delta\downarrow0$ in Euler's equation \eqref{euler} and obtain (recall \eqref{dv3}, \eqref{dv5} and $\ud\rightharpoondown\overline{u}$ in $L^{\zeta}(\Omega-D)^M$)
\begin{align}
\label{dv10}  \int\limits_{\Omega}{\tau:\n\varphi dx}+\lambda\int\limits_{\Omega-D}{|\overline{u}-f|^{\zeta-2}(\overline{u}-f)\cdot\varphi dx}=0
\end{align}
for all $\varphi\in W^{1,2}(\Omega)^M\cap L^{\zeta}(\Omega-D)^M$ and by approximation, equation \eqref{dv10} extends to $\varphi\in W^{1,1}(\Omega)^M\cap L^{\zeta}(\Omega-D)^M$ (see \cite{FT}, Lemma 2.1).\\
Besides we have
\begin{align*}
  R[\tau]:&=\inf\limits_{v\in W^{1,1}(\Omega)^M\cap L^{\zeta}(\Omega-D)^M}{l(v,\tau)}\\
&=\int\limits_{\Omega}{(-F^*(\tau))dx}\\
&+\inf\limits_{v}{\bigg[\int\limits_{\Omega}{\tau :\n v dx}+\frac{\lambda}{\zeta}\int\limits_{\Omega-D}{|v-f|^{\zeta}dx}\bigg]}\\
&=\int\limits_{\Omega}{(-F^*(\tau))dx}\\
&+\inf\limits_{v}{\bigg[-\lambda\int\limits_{\Omega-D}{|\overline{u}-f|^{\zeta-2}(\overline{u}-f)\cdot v\,dx}+\frac{\lambda}{\zeta}\int\limits_{\Omega-D}{|v-f|^{\zeta}dx}\bigg]}\\
&=\int\limits_{\Omega}{(-F^*(\tau))dx}-\lambda\int\limits_{\Omega-D}{|\overline{u}-f|^{\zeta-2}(\overline{u}-f)\cdot f\,dx}\\
&+\inf\limits_{v}{\bigg[-\lambda\int\limits_{\Omega-D}{|\overline{u}-f|^{\zeta-2}(\overline{u}-f)\cdot (v-f)\,dx}+\frac{\lambda}{\zeta}\int\limits_{\Omega-D}{|v-f|^{\zeta}dx}\bigg]}
 \end{align*}
where we have exploited that $\varphi=v$ serves as an admissible choice in equation \eqref{dv10}.\\
\newline
Obviously we can state
\begin{align*}
\inf\limits_{v\in W^{1,1}(\Omega)^M\cap L^{\zeta}(\Omega-D)^M}{[\ldots]}\geq \inf\limits_{v\in L^{\zeta}(\Omega-D)^M}{[\ldots]}
\end{align*}
since $[\ldots]=[\ldots]$ on $L^{\zeta}(\Omega-D)^M$. Quoting \eqref{f}, $|\overline{u}-f|^{\zeta-2}(\overline{u}-f)\in L^{\frac{\zeta}{\zeta-1}}(\Omega-D)^M$ (recall $\overline{u}\in L^{\zeta}(\Omega-D)^M$) and using H\"older's inequality we arrive at
\begin{align*}
&-\lambda\int\limits_{\Omega-D}{|\overline{u}-f|^{\zeta-2}(\overline{u}-f)\cdot (v-f)\,dx}+\frac{\lambda}{\zeta}\int\limits_{\Omega-D}{|v-f|^{\zeta}dx}\\
&\geq -\lambda\bigg(\int\limits_{\Omega-D}{|\overline{u}-f|^{\zeta}dx}\bigg)^{\frac{\zeta-1}{\zeta}}\bigg(\int\limits_{\Omega-D}{|v-f|^{\zeta}dx}\bigg)^{\frac{1}{\zeta}}\\
&+\frac{\lambda}{\zeta}\int\limits_{\Omega-D}{|v-f|^{\zeta}dx}.\\
\end{align*}
As the next step we use Young's inequality by choosing $\varepsilon:=1, p:=\frac{\zeta}{\zeta-1}>1, q:=\zeta>1$ (note $\frac{1}{p}+\frac{1}{q}=1$) and obtain 
\begin{align*}
&-\lambda\bigg(\int\limits_{\Omega-D}{|\overline{u}-f|^{\zeta}dx}\bigg)^{\frac{\zeta-1}{\zeta}}\bigg(\int\limits_{\Omega-D}{|v-f|^{\zeta}dx}\bigg)^{\frac{1}{\zeta}}+\frac{\lambda}{\zeta}\int\limits_{\Omega-D}{|v-f|^{\zeta}dx}\\
&\geq-\frac{\lambda (\zeta-1)}{\zeta}\int\limits_{\Omega-D}{|\overline{u}-f|^{\zeta}dx}-\frac{\lambda}{\zeta}\int\limits_{\Omega-D}{|v-f|^{\zeta}dx}+\frac{\lambda}{\zeta}\int\limits_{\Omega-D}{|v-f|^{\zeta}dx}\\
&=-\frac{\lambda (\zeta-1)}{\zeta}\int\limits_{\Omega-D}{|\overline{u}-f|^{\zeta}dx}
\end{align*}
for all $v\in L^{\zeta}(\Omega-D)^M$. This leads to
\begin{align}
\label{dv11} 
\begin{split}
R[\tau]\geq&\int\limits_{\Omega}{(-F^*(\tau))dx}-\lambda\int\limits_{\Omega-D}{|\overline{u}-f|^{\zeta-2}(\overline{u}-f)\cdot f\,dx}\\
&-\frac{\lambda (\zeta-1)}{\zeta}\int\limits_{\Omega-D}{|\overline{u}-f|^{\zeta}dx}.
\end{split}
\end{align}
In accordance with \eqref{dv9} we finally have shown
\begin{align}
 \label{dv13} \begin{split}
\sup\limits_{\rho\in L^{\infty}(\Omega)^{nM}}{R[\rho]}&\leq\inf\limits_{v\in W^{1,1}(\Omega)^M\cap L^{\zeta}(\Omega-D)^M}{I[v]}\leq R[\tau].
\end{split}
\end{align}
Thus, $\tau$ is $R$-maximizing and the inf-sup relation is valid. These facts prove assertion (a) of Theorem \ref{dual}.\\
Additionally, by virtue of \eqref{dv13}, we have shown that 
\begin{align}
\label{dv14}  \delta\int\limits_{\Omega}{|\nabla u_{\delta}|^2dx}\rightarrow0\\
\label{dv15} (u_{\delta})\,\,\text{is an}\,\,I-\text{minimizing sequence}
\end{align}
at least for a subsequence $\delta_m\rightarrow0$. Furthermore, it follows (see Theorem \ref{BV}, (d) and \eqref{dv15})
\begin{align}
 \label{dv16} \overline{u}\,\,\text{is}\,\,K-\text{minimizing in}\,\,BV(\Omega)^M\cap L^{\zeta}(\Omega-D)^M.
\end{align}
For assertion (b) we may proceed exactly as in \cite{BF2}, proof of Theorem 1.4. Since we have uniqueness the convergences \eqref{dv5} and \eqref{dv14} hold for any sequence $\delta\downarrow0$. Hence, the proof of Theorem \ref{dual} is complete.\hfill $\square$
\end{section}

\begin{section}{Uniqueness of the dual solution and the duality formula. Proof of Theorem \ref{Eindeutigkeit}}
Let the hypotheses of Theorem \ref{Eindeutigkeit} hold throughout this section and fix an arbitrary $K$-minimizer $u$ from the space $BV(\Omega)^M\cap L^{\zeta}(\Omega-D)^M$, whose existence is guaranteed by Theorem \ref{BV} (a). In case $\zeta=2$, the uniqueness of the dual solution $\sigma$ under the assumptions of Theorem \ref{dual} has already been proven in \cite{FT} (compare the proof of Theorem 1.3 in this reference). In this paper, one of the main arguments for proving the uniqueness of the dual solution relies on the verification that the special tensor $\rho:=DF(\n^a u)$ is a maximizer of the dual problem where (remember the Lebesgue decomposition $\n u=\n^au\llcorner\mathcal{L}^n+\n^su$) the density $\n^au$ is independent of the particular minimizer $u$. By definition, $\rho$ takes almost all of its values in the open and convex set 
\begin{align*}
U:=\text{Im}(DF),
\end{align*} 
thus, Theorem 2.13 from \cite{Bi1} applies to the special tensor $\rho$. At this point we can adopt the same strategy as used during the proof of Theorem 2.15 in \cite{Bi1} in order to show that $\rho$ represents the only dual solution.\\
\newline
Now let us consider the general case $\zeta>1$. In order to prove Theorem \ref{Eindeutigkeit} we adopt the procedure from the proof of Theorem 1.3 in \cite{FT} that has been described above. Recalling that the dual functional $R$ is given by
\[
 R[\varkappa]=\inf\limits_{v\in W^{1,1}(\Omega)^M\cap L^{\zeta}(\Omega-D)^M}{l(v,\varkappa)},\quad\varkappa\in L^{\infty}(\Omega)^{nM}
\]
where for all $(v,\varkappa)\in (W^{1,1}(\Omega)^M\cap L^{\zeta}(\Omega-D)^M, L^{\infty}(\Omega)^{nM})$ the Lagrangian $l(v,\varkappa)$ is defined through
\begin{align*}
l(v,\varkappa):=\int\limits_{\Omega}{(\varkappa:\n v-F^*(\varkappa))dx}+\frac{\lambda}{\zeta}\int\limits_{\Omega-D}{|v-f|^{\zeta}dx},
\end{align*}
we first assert
\begin{Lem}\label{maximierer}
The tensor $\sigma_0:=DF(\n^a u)$ is a maximizer of the dual functional $R$.
\end{Lem}
\begin{proof}[Proof of Lemma \ref{maximierer}]
On account of \eqref{v2} we observe that $\sigma_0$ belongs to the space $L^{\infty}(\Omega)^M$, i.e. $\sigma_0$ serves as an admissible choice in the Lagrangian $l(v,\varkappa)$ from above. For $v\in W^{1,1}(\Omega)^M\cap L^{\zeta}(\Omega-D)^M$ and $\varkappa=\sigma_0$ we then obtain
\[
 l(v,\sigma_0)=\int\limits_{\Omega}{DF(\nabla ^a u):\n v-F^*(DF(\nabla ^a u)))dx}+\frac{\lambda}{\zeta} \int\limits_{\Omega-D}{|v-f|^{\zeta}dx},
\]
which, by means of the duality relation $F(P)+F^*(DF(P))=P: DF(P), P\in \mathbb{R}^{nM}$, can be written in the form
\begin{align}
 \label{lag} \begin{split}
l(v,\sigma_0)=&\int\limits_{\Omega}{F(\nabla ^a u) dx}+\int\limits_{\Omega}{(\nabla v-\nabla^a u): DF(\nabla^a u)dx}\\
&+\frac{\lambda}{\zeta} \int\limits_{\Omega-D}{|v-f|^{\zeta}dx}.
 \end{split}
\end{align}
Exploiting that $u$ minimizes the functional $K$ (see \eqref{K}) we arrive at
\begin{align}
 \label{et1}
\begin{split} 0=\frac{d}{dt|_{0}} K[u+tv]=&\int\limits_{\Omega}{DF(\nabla ^a u): \nabla v dx}\\
&+\lambda \int\limits_{\Omega-D}{v\cdot(u-f)|u-f|^{\zeta-2}dx}
\end{split}
\end{align}
where we have used the relation $\n^s(u+tv)=\n^s u$ being valid for the singular parts of the measures. Evidently we further have $\n(u+tu)=(1+t)\n u$ and the $K$-minimality of $u$ additionally yields
\begin{align}
 \label{et2}
\begin{split}
 0=\frac{d}{dt|_{0}} K[u+tu]=&\int\limits_{\Omega}{DF(\nabla ^a u): \nabla^a u dx}+\int\limits_{\Omega}{F^{\infty}\bigg(\frac{\nabla ^s u}{|\nabla ^s u|}\bigg)d|\nabla^s u|}\\
&+\lambda \int\limits_{\Omega-D}{u\cdot(u-f)|u-f|^{\zeta-2}dx}.
\end{split}
\end{align}
Inserting \eqref{et1} and \eqref{et2} in \eqref{lag} it follows
\begin{align}
 \label{et3} \begin{split}
l(v,\sigma_0)&=\int\limits_{\Omega}{F(\nabla ^a u) dx}+\int\limits_{\Omega}{F^{\infty}\bigg(\frac{\nabla ^s u}{|\nabla ^s u|}\bigg)d|\nabla^s u|}\\
&-\lambda \int\limits_{\Omega-D}{v\cdot(u-f)|u-f|^{\zeta-2}dx}+\lambda \int\limits_{\Omega-D}{u\cdot(u-f)|u-f|^{\zeta-2}dx}\\
&+\frac{\lambda}{\zeta} \int\limits_{\Omega-D}{|v-f|^{\zeta}dx}\\
&=\int\limits_{\Omega}{F(\nabla ^a u) dx}+\int\limits_{\Omega}{F^{\infty}\bigg(\frac{\nabla ^s u}{|\nabla ^s u|}\bigg)d|\nabla^s u|}+\lambda\int\limits_{\Omega-D}{|u-f|^{\zeta}dx}\\
&-\lambda\int\limits_{\Omega-D}{(v-f)\cdot(u-f)|u-f|^{\zeta-2}dx}+\frac{\lambda}{\zeta} \int\limits_{\Omega-D}{|v-f|^{\zeta}dx}.
\end{split}
\end{align}
Now we proceed by considering the integral
\begin{align*}
-\lambda\int\limits_{\Omega-D}{(v-f)\cdot(u-f)|u-f|^{\zeta-2}dx}
\end{align*}
which appears on the right hand side of \eqref{et3}. Recalling that the function $(u-f)|u-f|^{\zeta-2}$ is of class $L^{\frac{\zeta}{\zeta-1}}(\Omega-D)^M$ (remember $u,f\in L^{\zeta}(\Omega-D)^M$) and applying H\"older's inequality along with Young's inequality (note that $\zeta>1$) we may estimate
\begin{align}
 \label{et4} \begin{split}
&-\lambda\int\limits_{\Omega-D}{(v-f)\cdot(u-f)|u-f|^{\zeta-2}dx}\\
&\geq-\lambda\bigg(\int\limits_{\Omega-D}{|u-f|^{\zeta}dx}\bigg)^{\frac{\zeta-1}{\zeta}}\bigg(\int\limits_{\Omega-D}{|v-f|^{\zeta}dx}\bigg)^{\frac{1}{\zeta}}\\   
&\geq-\frac{\lambda(\zeta-1)}{\zeta}\int\limits_{\Omega-D}{|u-f|^{\zeta}dx}-\frac{\lambda}{\zeta}\int\limits_{\Omega-D}{|v-f|^{\zeta}dx}.  
\end{split}
\end{align}
Combining \eqref{et4} and \eqref{et3} we obtain
\begin{align*}
l(v,\sigma_0)\geq K[u].
\end{align*}
Remembering the definition of the dual functional $R$ we may derive by using the $K$-minimality of $u$, Theorem \ref{BV} (c) and Theorem \ref{dual} (a)
\[
 R[\sigma_0]\geq K[u]=\inf\limits_{BV(\Omega)^M\cap L^{\zeta}(\Omega-D)^M}{K}=\inf\limits_{W^{1,1}(\Omega)^M\cap L^{\zeta}(\Omega-D)^M}{I}=\sup\limits_{L^{\infty}(\Omega)^{nM}}{R},
\]
whence we get that $\sigma_0$ is a $R$-maximizer taking almost all of its values within the open and convex set $U=\text{Im}(DF)$ in addition. This completes the proof of Lemma \ref{maximierer} . 
\end{proof}
For proceeding with the proof of Theorem \ref{Eindeutigkeit} we assume that the dual problem admits a second solution $\widetilde{\sigma}$ satisfying $\widetilde{\sigma}\neq\sigma_0$ on a set with positive measure. By means of this assumption we adopt the same arguments as used in the course of the proof of Theorem 2.15 in \cite{Bi1} and obtain the strict inequality
\[
 \int\limits_{\Omega}{(-F^*)\bigg(\frac{\widetilde{\sigma}+\sigma_0}{2}\bigg)dx}>\frac{1}{2}\int\limits_{\Omega}{(-F^*)(\widetilde{\sigma})dx}+\frac{1}{2}\int\limits_{\Omega}{(-F^*)(\sigma_0)dx}
\]
while at the same time we state that the function
\[
 L^{\infty}(\Omega)^{nM}\ni\varkappa\mapsto \inf\limits_{v\in W^{1,1}(\Omega)^M\cap L^{\zeta}(\Omega-D)^M}{\int\limits_{\Omega}[\varkappa:\n v-\mathbb{1}_{\Omega-D}|v-f|^{\zeta}]dx}
\]
is concave. Thus
\[
 R\bigg[\frac{\widetilde{\sigma}+\sigma_0}{2}\bigg]>\frac{1}{2}R[\widetilde{\sigma}]+\frac{1}{2}R[\sigma_0],
\]
which is a contradiction.\\
As a consequence, the tensor $DF(\n^au)$ represents the only maximizer of the dual problem and this completes the proof of Theorem \ref{Eindeutigkeit}.\hfill $\square$

\end{section}

\begin{section}{Partial regularity of generalized minimizers. Proof of Theorem \ref{partialregularity}}
Let us assume the validity of the hypotheses of Theorem \ref{partialregularity}. For proving the assertion of Theorem \ref{partialregularity}, our strategy is to use Corollary 3.3 in \cite{S} in the case $p=2$ therein. In order to justify an application of this corollary to our situation we first formulate a proposition which shows that the assumptions of Corollary 3.3 in \cite{S} are satisfied in our setting.
\begin{Prop}\label{Proppartial}
Assuming the hypotheses of Theorem \ref{partialregularity} it holds
\begin{enumerate}[(a)]
 \item For all $P\in\R^{nM}$ the density $F$ satisfies (H1)-(H4) in \cite{S} (see Section 2 in this reference).
\item Setting $g:\Omega\times\R^{M}\rightarrow\R,\,\, g(x,y):= \frac{\lambda}{\zeta}\mathbb{1}_{\Omega-D}|y-f(x)|^{\zeta}$ the following statements hold true:
\begin{enumerate}[(i)]
 \item $g$ is a Borel function;
\item $g$ satisfies a H\"older condition in the following sense: for a positive constant $C$, $0<\beta\leq1$ and $\beta\leq\zeta<\infty$ we have the estimate 
\begin{align}
 \label{hc} |g(x,y_1)-g(x,y_2)|\leq C(|f(x)|+|y_1|+|y_2|)^{\zeta-\beta}|y_2-y_1|^{\beta}
\end{align}
for all $x\in\Omega, y_1,y_2\in\R^M$;
\end{enumerate}
\item Each minimizer $u\in BV(\Omega)^M\cap L^{\zeta}(\Omega-D)^M$ of the functional (recall that $\overline{c}:=\lim\limits_{t\rightarrow\infty}{\frac{\Phi(t)}{t}}$ exists in $(0,\infty)$ since $\Phi$ is of linear growth)
\[
 K[w]=\int\limits_{\Omega}{\Phi(|\n^a w|)dx}+\overline{c}|\n^s w|(\Omega)+\frac{\lambda}{\zeta}\int\limits_{\Omega-D}{|w-f|^{\zeta}dx}
\]
satisfies 
\begin{align}
 \label{maxprin} \sup\limits_{\Omega}{|u|}\leq L
\end{align}
where $L:=\sup\limits_{\Omega-D}{|f|}$, i.e. each $K$-minimizer $u$ actually is of class $BV(\Omega)^M\cap L^{\infty}(\Omega)^M$.
\end{enumerate}
\end{Prop}

\begin{proof}[Proof of Proposition \ref{Proppartial}]
On account of \eqref{bs} with $\Phi$ satisfying $(1.3^*)-(1.4^*)$ and $(1.6^*)$ from \cite{BF2} we first state that $F$ satisfies \eqref{v1s}--\eqref{v4}. For proving assertion (a) we note that on account of \eqref{v4} we deal with the non-degenerate case. Quoting Remark 2.6 in \cite{S} we then choose $p=2$ in this reference and as a consequence (H2), (H3) as well as (H4) in \cite{S} correspond to requiring that the density $F$ is of class $C^2(\R^{nM})$ where $D^2F(P)$ is positive for all $P\in\R^{nM}$. Thus, $F$ satisfies (H2)-(H4) in \cite{S} by recalling \eqref{v4}. Furthermore, $F$ fulfills (H1) since $F$ is (strictly) convex on $\R^{nM}$ (see \eqref{v4}) and of linear growth in the sense of \eqref{lg}.\\
\newline
In order to verify the statements of part (b) we first remark that assertion (b), (i) is immediate whereas a calculation of $\n_yg(x,y)$ shows that \eqref{hc} holds with $\beta:=\text{min}(1,\zeta)=1$.\\
\newline
Finally we establish assertion (c): denoting by $u\in BV(\Omega)^M\cap L^{\zeta}(\Omega-D)^M$ an arbitrary $K$-minimizer whose existence is guaranteed by Theorem \ref{BV} we consider the projection (recall $L:=\sup\limits_{\Omega-D}{|f|}$)
\begin{align*}
H:\,\,\R^M&\rightarrow\R^M\\
 y&\mapsto
\left\{
  \begin{array}{ll}
    y, &|y|\leq L \\
    L\frac{y}{|y|}, & |y|>L
  \end{array}
\right..
\end{align*}
Next we define $v:=H\circ u$ and by quoting \cite{AFP}, Theorem 3.89 and the comments that are given at the beginning of its proof, respectively, it holds $v\in BV(\Omega)^M$ and we get the important inequality
\begin{align}
 \label{mp1} |\n v|\leq\,\text{Lip}(H)|\n u|=|\n u|.
\end{align}
Further we even have $v\in L^{\infty}(\Omega)^M$ and due to the minimality of $u$ we can state
\begin{align}
 \label{mp2} K[u]\leq K[v].
\end{align}
for all $v\in BV(\Omega)^M\cap L^{\zeta}(\Omega-D)^M$. Moreover we observe the validity of
\begin{align}
\label{mp3} |v-f|^{\zeta}\leq |u-f|^{\zeta}\quad\text{a.e. on}\,\,\Omega-D
\end{align}
after performing some straight-forward calculations.\\
As the next step we define the functional $\widetilde{K}$ through 
\begin{align*}
 \widetilde{K}[w]:=\int\limits_{\Omega}{\Phi(|\n^a w|)dx}+\overline{c}|\n^sw|(\Omega)
\end{align*}
being well-defined for functions $w\in BV(\Omega)^M$ (recall in addition that $\Phi$ is of linear growth).\\
Now we adopt the arguments from \cite{BF6} by starting with (6) in this reference (we essentially use the inequality \eqref{mp1}) and obtain
\begin{align}
\label{mp4} \int\limits_{\Omega}{\Phi(|\n^a v|)dx}&\leq \int\limits_{\Omega}{\Phi(|\n^a u|)dx},\\
\label{mp5} |\n^sv|(\Omega)&\leq |\n^su|(\Omega).
\end{align}
Hence using \eqref{mp3}, \eqref{mp4} and \eqref{mp5} together with \eqref{mp2} we obtain
\begin{align*}
 K[u]=K[v] 
\end{align*}
being only possible if
\begin{align}
\label{mp6} \int\limits_{\Omega}{\Phi(|\n^a u|)dx}&=\int\limits_{\Omega}{\Phi(|\n^a v|)dx},\\
\label{mp7} |\n^su|(\Omega)&=|\n^sv|(\Omega),\\
\label{mp8} \int\limits_{\Omega-D}{|u-f|^{\zeta}dx}&=\int\limits_{\Omega-D}{|v-f|^{\zeta}dx}.
\end{align} 
The identities \eqref{mp6} and \eqref{mp7} may then be exploited as done in \cite{BF6} in order to derive $\n u=\n v$ (we start with (13) in this reference). Further, \eqref{mp8} gives $u=v$ a.e. on $\Omega-D$ and using \cite{AFP}, Proposition 3.2, p.118 together with \eqref{L} we finally get $u=v$ a.e. on $\Omega$. In conclusion we have the validity of \eqref{maxprin} and the proof of Proposition \ref{Proppartial} is complete.
\end{proof}

\begin{Rem}
Following the lines of the proof of Proposition \ref{Proppartial} (c) it becomes evident that we need the structure condition \eqref{bs} on $F$ as well as an $L^{\infty}$-condition on $f$ (see \eqref{funendlich}). Of course, Proposition \ref{Proppartial} (c) remains valid under much weaker assumptions on the density $F$. In fact we only need that $F$ fulfills the structure condition \eqref{bs} with a strictly increasing and convex function $\Phi:[0,\infty)\rightarrow[0,\infty)$ being of linear growth (see \cite{BF6}, Theorem 1).
\end{Rem}

Now we proceed with the proof of Theorem \ref{partialregularity}. As in Proposition \ref{Proppartial} we denote by $u$ any $K$-minimizer from the space $BV(\Omega)^M\cap L^{\zeta}(\Omega-D)^M$. In order to apply Corollary 3.3 in \cite{S} to our situation we still need to check under what assumptions on the number $0<\alpha\leq\beta(=1)$ the following Morrey assumption holds true
\begin{align}
 \label{mo} |u|,f \in L^{(\zeta-1)n,\alpha n}(\Omega).
\end{align}
Here, for fixed numbers $q,\chi\in[0,\infty)$ the space $L^{q,\chi}(\Omega)$ denotes the so-called Morrey space (see, e.g., \cite{S}, Definition 4.9, for a definition of these spaces). Following the comments after the statement of Corollary 3.3 in \cite{S} the above Morrey condition \eqref{mo} is strongest if we let $\alpha=\beta$ (we then have the identity $L^{q,n}(\Omega)=L^{\infty}(\Omega)$ for $q>0$). In this case, \eqref{mo} represents a $L^{\infty}$-assumption and by quoting \eqref{funendlich} as well as the maximum principle stated in Proposition \ref{Proppartial} being valid for each $K$-minimizer $u$, the Morrey assumption \eqref{mo} is satisfied.\\
At this point we like to emphasize that the choice $\alpha=1$ is optimal since it holds $L^{q,\chi}(\Omega)=\{0\}$ for $\chi>n$.\\
\newline
Considering $\alpha<\beta$, \eqref{mo} is weaker and can be deduced from the condition $|u|,f\in L^{\frac{\zeta-1}{1-\alpha}n}(\Omega)$ in the context of Lebesgue spaces. Of course, on account of \eqref{funendlich} and Proposition \ref{Proppartial} (c), the partial observation $f$ and each $K$-minimizer $u$ satisfy the Morrey assumption \eqref{mo} in the case $\alpha<\beta$.\\
\newline
Consequently, all assumptions from Corollary 3.3 in \cite{S} are satisfied with optimal value $\alpha=1$ and it follows that for each $K$-minimizer $u$ from the space $BV(\Omega)^M\cap L^{\zeta}(\Omega-D)^M$ there exists an open subset $\Omega^u_0$ of $\Omega$ such that $u\in C^{1,\gamma}(\Omega^u_0)$ for all $0<\gamma\leq\widetilde{\gamma}:=\frac{\alpha}{2}=\frac{1}{2}$ with $\mathcal{L}^n(\Omega-\Omega^u_0)=0$. This proves Theorem \ref{partialregularity}.\hfill $\square$
 
\begin{Rem}
We conjecture that the choice of the exponent $\widetilde{\gamma}=\frac{1}{2}$ is not optimal although in the context of Corollary 3.3 in \cite{S}, this choice seems to be optimal. Moreover, we are not sure whether one needs the structure condition \eqref{bs} for the density $F$ in order to prove almost everywhere regularity of arbitrary $K$-minimizers on the entire domain $\Omega$. Because of the presence of the data fitting term it is not possible to refer to , e.g. Theorem 1.1 in \cite{AG} and adding some obvious modifications.
\end{Rem}

\end{section}

\end{document}